\documentclass[12pt]{amsart}
\usepackage{amssymb,latexsym,amsmath,amsthm,amscd}
\usepackage{setspace}
\usepackage{verbatim}
\usepackage[active]{srcltx}
\usepackage[all]{xy} \xyoption{arc}
\usepackage[left=3cm,top=2cm,right=3cm,bottom = 2cm]{geometry}
\usepackage{graphicx}
\usepackage[usenames,dvipsnames]{color}
\usepackage{charter}

\theoremstyle{plain}
\newtheorem{thm}{Theorem}[section]
\newtheorem{lem}[thm]{Lemma}
\newtheorem{prop}[thm]{Proposition}
\newtheorem{cor}[thm]{Corollary}

\theoremstyle{definition}
\newtheorem{dfn}[thm]{Definition}
\newtheorem{ex}[thm]{Example}

\theoremstyle{remark}
\newtheorem{rmk}[thm]{Remark}

\newcommand{\cF}{\mathcal{F}}

\newcommand{\cL}{\mathcal{L}}
\newcommand{\cM}{\mathcal{M}}
\newcommand{\cN}{\mathcal{N}}
\newcommand{\cO}{\mathcal{O}}

\newcommand{\cS}{\mathcal{S}}

\newcommand{\cU}{\mathcal{U}}
\newcommand{\cV}{\mathcal{V}}
\newcommand{\cW}{\mathcal{W}}

\newcommand{\frakh}{\mathfrak{h}}

\DeclareMathOperator{\uhp}{\mathcal{H}}
\DeclareMathOperator{\rk}{rk}

\DeclareMathOperator{\Tr}{Tr}

\DeclareMathOperator{\Hom}{Hom}
\DeclareMathOperator{\Spec}{Spec}

\DeclareMathOperator{\Pic}{Pic}
\DeclareMathOperator{\GL}{GL}

\DeclareMathOperator{\SL}{SL}

\DeclareMathOperator{\id}{id}
\DeclareMathOperator{\diag}{diag}
\newcommand*{\df}{\mathrel{\vcenter{\baselineskip0.5ex \lineskiplimit0pt
                     \hbox{\scriptsize.}\hbox{\scriptsize.}}} =}

\providecommand{\twomat}[4]{\left(\begin{matrix}#1&#2\\#3&#4\end{matrix}\right)}
\providecommand{\stwomat}[4]{\left(\begin{smallmatrix}#1&#2\\#3&#4\end{smallmatrix}\right)}

\newcommand{\CC}{\mathbf{C}}

\newcommand{\ZZ}{\mathbf{Z}}
\newcommand{\PP}{\mathbf{P}}
\newcommand{\RR}{\mathbf{R}}

\newcommand{\DD}{\mathbf{D}}

\newcommand{\cMbar}{\overline{\cM}}
\newcommand{\cLbar}{\overline{\cL}}
\newcommand{\cVbar}{\overline{\cV}}

\newcommand{\cObar}{\cO}

\newcommand{\cSbar}{\overline{\cS}}

\newcommand{\dbls}{/ \hspace{-1.5mm}/\hspace{-.5mm}}
\newcommand{\orb}[2]{#1\backslash\!\!\backslash #2}

\begin{document}
\title{vector valued modular forms and the modular orbifold of elliptic curves}
\author{Luca Candelori and Cameron Franc}
\date{}

\begin{abstract}
This paper presents the theory of holomorphic vector valued modular forms from a geometric perspective. More precisely, we define certain holomorphic vector bundles on the modular orbifold of generalized elliptic curves whose sections are vector valued modular forms. This perspective simplifies the theory, and it clarifies the role that exponents of representations of $\SL_2(\ZZ)$ play in the holomorphic theory of vector valued modular forms. Further, it allows one to use standard techniques in algebraic geometry to deduce free-module theorems and dimension formulae (deduced previously by other authors using different techniques), by identifying the modular orbifold with the weighted projective line $\PP(4,6)$. 
\end{abstract}
\maketitle

\tableofcontents
\setcounter{tocdepth}{1}

\section{Introduction}

Vector valued modular forms have played a role in number theory \cite{Borcherds}, \cite{EichlerZagier}, \cite{Selberg}, along with areas in mathematical physics, for some time now. A systematic treatment of their theory has only been initiated in recent years by Bantay, Gannon  \cite{Bantay},  \cite{BantayGannon}, \cite{Gannon}, Knopp, Mason \cite{KnoppMason2}, \cite{Mason1}, \cite{Mason2} and others \cite{Marks}, \cite{MarksMason}, \cite{SaberSebbar}. Most of these approaches are based upon the Riemann-Hilbert correspondence, or vector valued Poincar\'{e} series. In this paper we present a new geometric perspective on the subject by viewing vector valued modular forms as sections of certain vector bundles over the modular orbifold of generalized elliptic curves\footnote{The paper \cite{SaberSebbar} also defines vector valued modular forms as sections of certain vector bundles on Riemann surfaces. We discuss how their work compares with ours below.}.

This geometric perspective was advocated by Gannon \cite{Gannon}. In particular, Gannon notes that there are results in the theory of vector valued modular forms that should follow from suitably generalized versions of the Birkhoff-Grothendieck Theorem, Riemann-Roch and Serre Duality (\cite{Gannon}, \S 3.3 and \S 3.5). In this paper we make these connections entirely rigorous by viewing the modular orbifold as the weighted projective line $\PP(4,6)$, for which analogs of the Birkhoff-Grothendieck Theorem (Theorem \ref{thm:DecompositionTheorem}, due to Meier and Vistoli \cite{Meier}), Riemann-Roch (Theorem \ref{thm:eulerFormula}, due to Edidin \cite{Edidin})  and Serre Duality (Proposition \ref{prop:weakSerreDuality} of this paper), are well-known.  

It is worth noting that many of the results presented below have been obtained by Gannon \cite{Gannon} using a different approach, and in a slightly more general context of admissible multiplier systems of weight $w\in\CC$. His approach builds on work of Borcherds \cite{Borcherds} and joint work between Bantay and Gannon \cite{BantayGannon}, and it makes essential use of the solution to the Riemann-Hilbert problem. In this paper, we restrict to the case when $w\in\ZZ$. However, it is entirely plausible that our methods can be applied to the study of more general admissible multiplier systems by replacing the modular orbifold with more general stacks, e.g. the stack $\cM_{1,\vec{1}}$ of \cite{Hain}, \S 8, whose fundamental group is the braid group on three strings. It is also worth noting that Gannon's results are often stated under the assumption that $\rho(T)$ (see below for notation) is diagonalizable. We do not need this assumption in this paper, since our methods do not require explicit computations with $q$-expansions (or $\log q$ expansions), which can be prohibitive when $\rho(T)$ is not diagonalizable.
 
In Section \ref{s:vvmfs} of this paper we define the vector bundles  $\cVbar_{k,L}(\rho)$ of vector valued modular forms of weight $k\in \ZZ$ for a finite-dimensional representation $\rho$ of $\SL_2(\ZZ)$ and choice of exponents $L$ (Definition \ref{d:standardexponents} below). The global sections of these vector bundles are precisely the spaces $\cM^\lambda_w(\rho)$ of \cite{Gannon}, \S 3.4 with $w=k$ and $\lambda=L$. With respect to these definitions, the main contribution of this paper, aside from the intrinsic interest of the modular perspective, is to clarify the role that the choice of exponents $L$ plays in the theory. For example, we explain how the choice of exponents relative to an interval $[0,1)$ yields holomorphic vector valued modular forms in the classical sense. The interval $(0,1]$ yields cusp forms, and intervals of the form $[\frac{a}{12},\frac{a}{12}+1)$ yield subspaces of holomorphic forms that are divisible by $\eta^{2a}$, where $\eta$ is the Dedekind eta function.

As an immediate application of this geometric definition of vector valued modular forms, we recover the well-known (e.g. \cite{Gannon}, \cite{MarksMason}) free-module theorem for vector valued modular forms (Theorem \ref{thm:FreeModuleTheorem} below). Our statement mirrors that of \cite{Gannon}, Theorem 3.4, without restrictions on $\rho(T)$, but with restrictions on the weight being integral. We also obtain simple formulas for the Euler characteristic of the vector bundles $\cVbar_{k,L}(\rho)$ (Corollary \ref{c:eulerchar}). In most cases, this formula is enough to also deduce dimension formulas for the vector spaces of holomorphic vector valued modular forms and cusp forms (Theorem \ref{t:dimension}). These results are obtained by identifying the compactified modular orbifold with the weighted projective line $\PP(4,6)$, thus viewing the vector bundles  $\cVbar_{w,\lambda}(\rho)$ as purely algebraic objects. We may then apply the above-mentioned results of Meier \cite{Meier} on weighted projective lines, and the Riemann-Roch theorem for $\PP(4,6)$ \cite{Edidin}. 

As it turns out, the isomorphism class of $\cVbar_{k,L}(\rho)$ is entirely determined by a $n$-tuple of integers, $n=\dim \rho$, which we call the {\em roots} of $\rho$ (Definition \ref{dfn:roots}). These are the negative of the `generating weights' of \cite{Gannon}. We devote Section \ref{s:examples} to computing these roots in a variety of examples.

Several other authors have discussed dimension formulae for spaces of vector valued modular forms. Most of these \cite{Bantay},\cite{Borcherds}, \cite{Freitag}, \cite{Skoruppa} restrict to representations of finite image. In \cite{Skoruppa} this restriction arises due to an application of a trace formula, while in \cite{Borcherds} and \cite{Freitag} this finiteness condition allows the authors to work on a finite cover of the modular orbifold that is in fact a scheme. In \cite{BantayGannon}, the authors assume $\rho(T)$ is of finite order, while Gannon (\cite{Gannon}, Lemma 3.2)  avoids imposing any  finiteness condition via an application of the solution to the Riemann-Hilbert problem, but assumes $\rho(T)$ diagonal. The present paper avoids imposing any finiteness or diagonalizability conditions via a technique modeled after the construction of extensions of a regular connection on a punctured sphere -- see \cite{Deligne}, \cite{PetersSteenbrink} and Proposition \ref{prop:extensionOfVValuedModForms} of the present paper.

The paper \cite{SaberSebbar} also describes vector valued modular forms as sections of vector bundles on noncompact Riemann surfaces, with conditions imposed at elliptic points and cusps. The authors of \cite{SaberSebbar} prove the existence of vector valued modular forms in great generality for arbitrary Fuchsian groups, and without the aid of stacks (which is why they must privilege the elliptic points). Their construction is modeled after the extension of a regular connection on a noncompact curve, as is ours. Due to the great generality of the paper, \cite{SaberSebbar} necessarily focuses on important basic questions such as the existence of modular forms. It does not touch on topics such as dimension formulae or free-module theorems\footnote{For general Fuchsian groups one might expect at best a projective-module theorem.}, and it lacks a moduli perspective. In contrast, in limiting ourselves to $\SL_2(\ZZ)$, the scope of our paper is narrower than \cite{SaberSebbar}, but we are able to go more deeply into the subject.

Let us finally note that our definitions are complex analytic, and we rely on a GAGA result for Deligne-Mumford stacks \cite{Toen} to inject results from algebraic geometry into the discussion. It would be of interest to provide a purely algebro-geometric construction of vector valued modular forms for as broad a class of representations as possible (for example, at least for representations of finite image), while working over an integral base such as $\ZZ[1/M]$. Such a perspective would lend insight into arithmetic questions about noncongruence modular forms, such as questions of unbounded denominators \cite{AtkinSwinnertonDyer}.

The following notation is used throughout the present paper: set
\begin{align*}
  T &= \twomat 1101, & S &= \twomat{0}{-1}{1}{0}, & R &= \twomat{0}{-1}{1}{1}.
\end{align*}
The function $\eta$ denotes the Dedekind eta function, and $\chi$ is the character of $\SL_2(\ZZ)$ corresponding to $\eta^2$. Thus $\chi(T) = e^{2\pi i \frac{1}{12}}$, $\chi(S) = -i$ and $\chi(R) = e^{2\pi i\frac{5}{6}}$.\footnote{Formula (12) of \cite{FrancMason} incorrectly reads $\chi(S) = i$, but this does not affect the results of that paper.} The notation $\chi(\cV)$, where $\cV$ is a vector bundle, will also be used to denote the Euler characteristic of $\cV$, but no confusion between the two notations should arise. Throughout this paper $\rho$ will denote a finite-dimensional complex representation of $\SL_2(\ZZ)$, and $\rho^\vee$ denotes the dual representation. If $\cV$ is a vector bundle, then $\cV^\vee$ denotes the dual vector bundle. If $H^i(X,\cF)$ denotes the cohomology of a sheaf $\cF$ on some space $X$, then $h^i(X,\cF)$ denotes the dimension of $H^i(X,\cF)$
whenever this makes sense.

The authors thank Dan Edidin, Terry Gannon, Geoffrey Mason and Lennart Meier for several helpful discussions. We also thank the referee for comments and corrections. This project began through collaboration at the mini-workshop on \emph{Algebraic Varieties, Hypergeometric series and Modular Forms} held at LSU in April 2015. The authors would like to acknowledge the organizers and the sponsors, Microsoft Research, the Number Theory Foundation, and the LSU Office of Research and Economic Development, for their support. 

\section{The modular orbifold of elliptic curves}
\label{section:modularForms}

Let $\frakh\df\{ z\in \CC : \mathrm{Im}(z)>0\}$ denote the complex upper half-plane and let
$$
\cM^{\rm an} \df \orb{\SL_2(\ZZ)}{\frakh}
$$
denote the {\em modular orbifold}, obtained by taking the quotient (in the category of orbifolds) of the action of $\SL_2(\ZZ)$ on $\frakh$ by linear fractional transformations. A detailed description of this orbifold can be found in \cite{Hain}, which is also our main reference for this section. For each integer $k\in \ZZ$, there is an action of $\SL_2(\ZZ)$ on $\CC\times \frakh$ defined as 
\begin{equation}
\label{eqn:modularFormsLineBundle}
  \twomat abcd (z,\tau) = \left((c\tau + d)^kz,\frac{a\tau + b}{c\tau + d}\right).
\end{equation}
The orbifold quotient $\orb{\SL_2(\ZZ)}{\CC\times\frakh}$ by this action defines a line bundle $\cL_k$ on $\cM^{\rm an}$, whose holomorphic sections are holomorphic functions $f:\frakh \rightarrow \CC$ satisfying
\begin{equation}
\label{eqn:modularFormsFormula}
f\left ( \frac{a\tau + b}{c\tau + d} \right) = (c\tau +d)^kf(\tau), \text{ for all } \twomat abcd \in \SL_2(\ZZ).
\end{equation}
That is, they are (level one, weakly holomorphic) modular forms of weight $k$.  

The orbifold $\cM^{\rm an}$ admits a canonical compactification  $\cMbar^{\rm an}$, which can be constructed as follows \cite{Hain}: consider the quotient $\orb{\langle -I_2, T \rangle}{\frakh}$, where $I_2$ is the identity matrix. This quotient is a covering
$$
\iota_1: \orb{\langle -I_2, T \rangle}{\frakh} \longrightarrow  \orb{\SL_2(\ZZ)}{\frakh} = \cM^{\rm an}
$$
of the modular orbifold $\cM^{\rm an}$. The subgroup $ \langle -I_2, T \rangle$ is isomorphic to $C_2\times \ZZ$ via $(\pm 1, n) \mapsto \pm T^n$, and it acts on $\frakh$ by $(\pm 1, n)\tau = \tau + n$. Consequently, the exponential map $\tau \mapsto q=e^{2\pi i\tau}$ defines an isomorphism of orbifolds $\orb{\langle -I_2, T \rangle}{\frakh} \cong \orb{C_2}{\DD^\times}$, where $\DD^\times$ is the punctured unit disk, and $C_2$ acts trivially on it.  We thus have a diagram 
\begin{equation}
\label{diagram:compactification}
\xymatrixcolsep{4pc}\xymatrix{
&\orb{\langle -I_2, T \rangle}{\frakh}\ar[ld]_{\iota_1}\ar@{<->}[r]^{\tau \mapsto e^{2\pi i\tau}} & \orb{C_2}{\DD^\times}\ar[rd]^{\iota_2} &\\
\orb{\SL_2(\ZZ)}{\frakh} = \cM^{\rm an} &&& \orb{C_2}{\DD},}
\end{equation}
where $\iota_2: \orb{C_2}{\DD^{\times}} \hookrightarrow \orb{C_2}{\DD}$ is induced by the canonical inclusion $\DD^{\times}\hookrightarrow \DD$. 

\begin{dfn}
The {\em compactified modular orbifold}, denoted $\cMbar^{\rm an}$, is the orbifold obtained by glueing $\cM^{\rm an}$ and $\orb{C_2}{\DD}$ along the maps $\iota_1$ and $\iota_2$ of diagram \eqref{diagram:compactification}.
\end{dfn}

The orbifold $\cMbar^{\rm an}$ can be thought of as being obtained from $\cM^{\rm an}$ by adding an orbifold point $\infty$ with automorphism group equal to $C_2$, corresponding to the origin of $\orb{C_2}{\DD}$. The point $\infty$ is called the \emph{cusp} of $\cMbar^{\rm an}$.

By descent for line bundles over orbifolds, a line bundle $\cN$ over $\cMbar^{\rm an}$ can be specified by giving a triple $(\cN_1,\cN_2,\phi)$ of a line bundle $\cN_1$ over $\cM^{\rm an}$, a line bundle $\cN_2$ over $\orb{C_2}{\DD}$, and a bundle isomorphism
$$
\phi: \iota_1^*\cN_1 \stackrel{\cong}\longrightarrow \iota_2^*\cN_2
$$  
lying over the map $\tau \mapsto e^{2\pi i\tau}$. The triple $(\cN_1,\cN_2,\phi)$ will be called an {\em extension} of $\cN_1$ to $\cMbar^{\rm an}$. There is a canonical extension of $\cL_k$, whose explicit construction we now recall since the same method will be applied in Section \ref{s:vvmfs} below. 

\begin{prop}[\cite{Hain}, Proposition 4.1]
\label{prop:canonicalExtensionOfModularForms}
The line bundle of modular forms $\cL_k$, defined by \eqref{eqn:modularFormsLineBundle}, has a canonical extension to $\cMbar^{\rm an}$, denoted by $\cLbar_k$, such that there is a canonical isomorphism $\cLbar_k \cong \cLbar_1^{\otimes k}$ for any integer $k$. 
\end{prop}

\begin{proof}
Let $\cN_2$ be the line bundle over $\orb{C_2}{\DD}$ given by the quotient $\orb{C_2}{\CC\times\DD}$ by the action $(\pm 1)(z,q) = \left((\pm 1)^kz,q\right)$. Then $\iota_2^*\cN_2$ is simply the quotient $\orb{C_2}{\CC\times\DD^{\times}}$ by the same action. On the other hand, if we let $\cN_1 = \cL_k$ be the line bundle  $\orb{\SL_2(\ZZ)}{\CC\times\frakh}$ given by \eqref{eqn:modularFormsLineBundle}, then $\iota_1^*\cL_k$ is the quotient $\orb{\langle \pm I_2, T\rangle }{\CC\times\frakh}$, where the action is the same as \eqref{eqn:modularFormsLineBundle}, but restricted to $\langle \pm I_2, T\rangle$. Finally, let $\phi$ be the map
\begin{align*}
\CC\times\frakh &\longrightarrow \CC\times\DD^{\times} \\
(z,\tau) &\longmapsto (z, e^{2\pi i \tau}).
\end{align*}
Then
$$
\phi( (\pm T^n)(z,\tau)) = \phi( (\pm 1)^kz,\tau+n) = ((\pm 1)^kz,e^{2\pi i\tau}) = (\pm 1)(z,e^{2\pi i\tau}) = (\pm 1)\phi(z,\tau),
$$
and thus $\phi$ gives a bundle map
$$
\orb{\langle \pm I_2, T\rangle }{\CC\times\frakh} \longrightarrow \orb{C_2}{\CC\times\DD^{\times}} 
$$
lying over $\tau \mapsto e^{2\pi i\tau}$. The canonical extension of $\cL_k$ is then given by the triple $(\cN_1=\cL_k,\cN_2,\phi)$. The statement about the compatibility with tensor products follows easily. 
\end{proof}

By \eqref{eqn:modularFormsFormula}, it is easy to see that $\cLbar_k$ has no global sections for odd integers $k$. Suppose then $k$ is even and let $f$ be a global section of $\cLbar_k$. We can restrict $f$ to $\cM^{\rm an}$ and then to $\orb{C_2}{\DD^{\times}}$, the punctured neighborhood of $\infty$, where we have
$
\iota_1^*f = \iota_2^*\widetilde{f}
$
for some section $\widetilde{f}$ of the line bundle $\cN_2$ of Proposition \ref{prop:canonicalExtensionOfModularForms}. But $\cN_2$ in this case is trivial, since $k$ is even, and thus  
$$
\widetilde{f} \in H^0(\orb{C_2}{\DD}, \cN_2) = H^0(\orb{C_2}{\DD}, \cO) = \CC[\![q]\!].
$$
In other words, the {\em $q$-expansion} $\widetilde{f}$ of $f$ contains only non-negative powers of $q$, that is, a global section $f$ of $\cLbar_k$ is a {\em holomorphic modular form} of weight $k$ and level one.

\section{Vector valued modular forms}
\label{s:vvmfs}
Let $\rho \colon \SL_2(\ZZ) \to \GL(V)$ denote a finite-dimensional complex representation of $\SL_2(\ZZ)$. Vector valued modular forms of weight $k$ for $\rho$ are holomorphic functions $F \colon \frakh \to V$ satisfying both the condition 
\begin{equation}
\label{eq:translaw}
  F(\gamma \tau) = (c\tau + d)^k\rho(\gamma)F(\tau), \text{ for all } \gamma = \stwomat abcd \in \SL_2(\ZZ),
\end{equation}
as well as a holomorphy condition at the cusp. Such functions were introduced as early as the 1950s (see \cite{Selberg}, for example), but their general study awaited the relatively recent work of Knopp and Mason \cite{KnoppMason1}, \cite{KnoppMason2}. This section describes the basic theory of vector valued modular forms in a basis-independent and geometric way, similar to the  description of holomorphic modular forms in Section \ref{section:modularForms}.

Let  $\cV_k(\rho) \df \orb{\SL_2(\ZZ)}{V\times\frakh}$ be the quotient of $V\times \frakh$ by the action 
\begin{equation}
\label{eqn:vvaluedAction}
\gamma (v,\tau) = \left((c\tau + d)^k\,\rho(\gamma)v,\frac{a\tau + b}{c\tau + d}\right) \text{ for all } \gamma = \twomat abcd \in \SL_2(\ZZ).
\end{equation}
The quotient $\cV_k(\rho)$ is a vector bundle over $\cM^{\rm an}$. When the representation $\rho$ is understood, we will often write $\cV_k$ in place of $\cV_k(\rho)$. Global holomorphic sections of $\cV_k \to \cM$ are holomorphic $V$-valued functions $F\colon \frakh \to V$ that transform as in (\ref{eq:translaw}).

In order to impose a holomorphy condition at the cusp on vector valued modular forms, one can extend $\cV_k$ to the compactified modular orbifold $\cMbar^{\rm an}$ as follows: using descent for vector bundles over orbifolds, the diagram \eqref{diagram:compactification} shows that an extension of $\cV_k$ to $\cMbar^{\rm an}$ is nothing but a triple $(\cW_1,\cW_2,\phi)$, where $\cW_1 = \cV_k$, $\cW_2$ is a vector bundle over $\orb{C_2}{\DD}$, and $\phi$ is a bundle isomorphism
$$
\phi: \iota_1^*\cW_1 \stackrel{\cong}\longrightarrow\iota_2^*\cW_2
$$
lying over $\tau \mapsto e^{2\pi i\tau}$. One can construct such extensions by using {\em exponent matrices}, defined as follows:

\begin{dfn}
Let $\rho:\SL_2(\ZZ)\rightarrow \GL(V)$ be a finite-dimensional representation. An endomorphism $L$ of $V$ is called an \emph{exponent matrix} for $\rho$ if $\rho(T) = e^{2\pi i L}$.
\end{dfn}

The following proposition is modelled after the canonical extension of a regular connection on an open curve, as discussed in \cite{Deligne}, \cite{PetersSteenbrink}, et cetera. See also \cite{BantayGannon}, \cite{Bantay}.

\begin{prop}
\label{prop:extensionOfVValuedModForms}
To each exponent matrix $L$ for $\rho$, there corresponds a unique extension of $\cV_k(\rho)$ to $\cMbar^{\rm an}$, denoted by $\cVbar_{k,L}(\rho)$.  
\end{prop}

\begin{proof}
We proceed as in Proposition \ref{prop:canonicalExtensionOfModularForms}. Let $\cW_2$ be the vector bundle over $\orb{C_2}{\DD}$ given by the quotient $\orb{C_2}{V\times\DD}$ by the action $(\pm 1)(v,q) = \left((\pm 1)^k\rho(\pm I_2)v,q\right)$. Then $\iota_2^*\cW_2$ is simply the quotient $\orb{C_2}{V\times\DD^{\times}}$ by the same action. 

Next let $\cW_1 \df \cV_k(\rho)$, so that $\iota_1^*\cW_1$ is the quotient $\orb{\langle \pm I_2, T\rangle }{V\times\frakh}$, where the action is given by equation \eqref{eqn:vvaluedAction}, restricted to $\langle \pm I_2, T\rangle$. Finally, let $\phi_L$ be the map
\begin{align*}
\phi_L: V\times\frakh &\longrightarrow V\times\DD^{\times} \\
(v,\tau) &\longmapsto (e^{-2\pi i\tau L}\,v, e^{2\pi i \tau}).
\end{align*}
One verifies easily that $\phi_L((\pm T^n)(v,\tau)) = (\pm 1)\phi_L(v,\tau)$, and thus $\phi_L$ gives a bundle isomorphism
$$
\orb{\langle \pm I_2, T\rangle }{V\times\frakh} \stackrel{\cong}\longrightarrow \orb{C_2}{V\times\DD^{\times}}
$$
lying over $\tau\mapsto e^{2\pi i\tau}$. Thus we may let $\cVbar_{k,L}(\rho)$ be the vector bundle over $\cMbar^{\rm an}$ defined by the triple $(\cW_1,\cW_2,\rho_L)$. 
\end{proof}

Proposition \ref{prop:extensionOfVValuedModForms} raises the question of when two extensions of $\cV_k$ to $\cMbar^{\rm an}$ are isomorphic. Again by descent, an isomorphism of two vector bundles $\cU,\cW$ over $\cMbar^{\rm an}$ corresponding to triples $(\cU_1, \cU_2, \psi)$, $(\cW_1,\cW_2,\phi)$ is given by a pair of isomorphisms
$$
\alpha_1: \cU_1\cong \cW_1, \quad \alpha_2: \cU_2\cong \cW_2
$$
over $\cM^{\rm an}$ and $\orb{C_2}{\DD}$, respectively, such that the following diagram is commutative:
\[
\xymatrixcolsep{4pc}\xymatrix{
\iota_1^*\cU_1\ar[r]^{\psi}\ar[d]_{\alpha_1} & \iota_2^*\cU_2\ar[d]^{\alpha_2}\\
\iota_1^*\cW_1\ar[r]^{\phi}&\iota_2^*\cW_2. }
\]

As a first example, we show that extending $\cV_k$ as in Proposition \ref{prop:extensionOfVValuedModForms} is canonically equivalent to extending $\cV_0$ and tensoring with $\cLbar_k$, provided that the choice of exponent matrix $L$ is the same for $\cV_k$ and $\cV_0$.

\begin{prop}
\label{prop:compatibilityOfExtensions}
Let $L$ denote an exponent matrix for $\rho$. Then there is a canonical isomorphism $\cVbar_{k,L}(\rho)\cong \cVbar_{0,L}(\rho)\otimes\cLbar_k$ of vector bundles over $\cM^{\rm an}$. 
\end{prop}

\begin{proof}
Let $(\cU_1, \cU_2, \psi)$ be the triple defining $\cVbar_{k,L}(\rho)$ as in the proof of Proposition \ref{prop:extensionOfVValuedModForms}, let  $(\cW_1,\cW_2,\phi)$ be the triple defining $\cVbar_{0,L}(\rho)$ and let $(\cN_1,\cN_2,\varphi)$ be the triple defining $\cLbar_k$ as in Proposition \ref{prop:canonicalExtensionOfModularForms}. Then the triple $(\cW_1\otimes\cN_1, \cW_2\otimes\cN_2,\phi\otimes\varphi)$ defines $\cVbar_{0,L}(\rho)\otimes\cLbar_k$. Now by \eqref{eqn:modularFormsLineBundle} and \eqref{eqn:vvaluedAction} the identity map gives an isomorphism 
$$
\id:\cU_1 = \cV_k(\rho)\stackrel{\cong}\longrightarrow \cW_1\otimes\cN_1 = \cV_0(\rho)\otimes\cL_k
$$
of vector bundles over $\cM$. Similarly, by looking at the proofs of Propositions \ref{prop:canonicalExtensionOfModularForms} and \ref{prop:extensionOfVValuedModForms} the identity gives an isomorphism 
$$
\id:\cU_2 \stackrel{\cong}\longrightarrow \cW_2\otimes\cN_2. 
$$
Finally, note that $\psi = \phi\otimes\varphi = \phi_L$, where $\phi_L$ is defined as in the proof of Proposition \ref{prop:extensionOfVValuedModForms}, and thus the pair $(\alpha_1 = \id, \alpha_2 = \id)$ is the required canonical isomorphism. 
\end{proof}

Next, we would like to compare the line bundles $\cVbar_{k,L}(\chi)$ obtained from characters $\chi:\SL_2(\ZZ)\rightarrow \CC^{\times}$ with the line bundles of modular forms $\cLbar_k$. This should indeed be possible, since it is a classical result (see theorem 6.9 of \cite{Hain}, for example) that $\Pic(\cMbar^{\rm an})\cong\ZZ$, where the class of $\cLbar_1$ generates $\Pic(\cMbar^{\rm an})$. 

Recall that the character group $\Hom(\SL_2(\ZZ),\CC^\times)$ is isomorphic with $\ZZ/12\ZZ$, with a generator given by the character $\chi$ of $\eta^2$. Since $\chi^a(T) = e^{2\pi i a/12}$ for $a  =0$, $\ldots$, $11$, extensions of $\cV_k(\chi^a)$ are determined by choices of exponent matrices $L=\frac{a}{12} + t$, for arbitrary $t\in\ZZ$. For these extensions we have:
\begin{thm}
\label{thm:characters}
Let $\cVbar_{k,L}(\chi^a)$ be the line bundle over $\cMbar^{\rm an}$ obtained from the character $\chi^a$ and the exponent matrix $L=a/12+t$, for some choice of $t\in\ZZ$. Then there is a canonical isomorphism
$$
\cVbar_{k,L}(\chi^a) \cong \cLbar_{k- a - 12t},
$$
for all $a=0,\ldots,11$ and all integers $k\in \ZZ$.   
\end{thm}

\begin{proof}
By Propositions \ref{prop:canonicalExtensionOfModularForms} and \ref{prop:compatibilityOfExtensions}, it suffices to show that
$\cVbar_{0,L}(\chi^a) \cong \cLbar_{-a-12t}$. To this end, let $(\cW_1 = \cV_0(\chi^a),\cW_2,\phi_L)$ be the triple defining $\cVbar_{0,L}(\chi^a)$ as in Proposition \ref{prop:extensionOfVValuedModForms} and let $(\cN_1 = \cL_{-a-12t},\cN_2,\phi)$ be the triple defining $\cLbar_{-a-12t}$ as in Proposition \ref{prop:canonicalExtensionOfModularForms}. Let
$$
\eta(\tau):= e^{2\pi i \tau/24}\displaystyle\prod_{n=1}^{\infty}(1-q^n), \quad q=e^{2\pi i\tau},
$$
be Dedekind's eta function. It is well-known that $\eta^{2(a+12t)}(\tau)$ is a non-vanishing section of $\cV_{a+12t}(\chi^a)$ over $\cM^{\rm an}$. Therefore division by $\eta^{2(a+12t)}(\tau)$ gives a trivialization $\cV_{a+12t}(\chi^a)\cong \cO$. Equivalently, since $\cV_{a+12t}(\chi^a)\cong \cV_0(\chi^a)\otimes\cL_{a+12t}$, we have an isomorphism
$
\alpha_1 : \cV_0(\chi^a)\stackrel{\cong}\longrightarrow \cL_{-a-12t}
$
of line bundles over $\cM^{\rm an}$, given by 
$$
\alpha_1(z,\tau) = (\eta^{-2(a+12t)}(\tau)\,z,\tau).
$$
On the other hand, since $\chi^a(-I_2) = (-1)^a = (-1)^{-a-12t}$, we have $\cW_2 = \cN_2$ as line bundles over the orbifold $\orb{C_2}{\DD}$. Thus any choice of $\alpha_2:\cW_2\cong\cN_2$ is just a line bundle automorphism, hence determined by multiplication by a unit in $\CC[\![q]\!]^{\times}$. We may thus let
$$
\alpha_2(z,q): = \left(\left(\prod_{n=1}^{\infty}(1-q^n)\right)^{-2(a+12t)}z,q\right),
$$ 
which is well-defined since $\prod_{n=1}^{\infty}(1-q^n)$ is a unit in $\CC[\![q]\!]$. The theorem then follows by noting that the diagram 
$$
\xymatrixcolsep{5pc}\xymatrix{
\iota_1^*\cV_0(\chi^a)\ar[r]^{\phi_L}\ar[d]_{\alpha_1} & \iota_2^*\cW_2\ar[d]^{\alpha_2}\\
\iota_1^*\cL_{-a-12t}\ar[r]^{\phi}&\iota_2^*\cN_2 }
$$
is commutative, since 
\begin{align*}
\alpha_2\circ\phi_L(z,\tau) &= \left(e^{-2\pi i\tau(a/12 + t)}\left(\displaystyle\prod_{n=1}^{\infty}(1-q^n)\right)^{-2(a+12t)}\,z,q\right)\\
&= (\eta^{-2(a+12t)}(\tau)\,z, q)\\
&= \phi\circ\alpha_1(z,\tau).
\end{align*}
\end{proof}

\begin{ex}
\label{ex:modFormsVsCuspForms}
If $\chi^a = \mathbf{1}$ is the trivial representation, then $\cV_k(\mathbf{1}) = \cL_k$. If we take the exponent $L = 0$, then $\cVbar_{k,0}(\mathbf{1}) = \cLbar_k$ by Theorem \ref{thm:characters}, thus sections of $\cVbar_k(\mathbf{1})$ are just holomorphic modular forms of weight $k$. On the other hand, if we choose $L=1$, then $\phi_L(z,q) = ((-1)^kq^{-1}z,q)$: if $f$ is a section of $\cVbar_{k,1}(\mathbf{1})$, $k$ even, then $q^{-1}\iota_1^*f \in \CC[\![q]\!]$, i.e. $f$ is a {\em cusp form} of weight $k$. Theorem \ref{thm:characters} then specializes to the well-known isomorphism
$$
\{ \text{cusp  forms of weight } k \}\stackrel{\cong}\longleftrightarrow \{ \text{hol. modular forms of weight } k-12 \}, 
$$
given by divison by $\Delta = \eta^{24}$. 
\end{ex}

The following properties will be used repeatedly in the sequel. In particular, (ii) of the following proposition will be used to compute the Euler characteristic of the vector bundles $\cVbar_{k,L}(\rho)$, while (iii) is used in the discussion of Serre-duality.
\begin{prop}
\label{p:functorialproperties}
Let $\rho$ and $\sigma$ denote representations of $\SL_2(\ZZ)$, where $\rho$ is of dimension $d$. Let $L$ and $L'$ denote choices of exponents for $\rho$ and $\sigma$, respectively. Then the following properties hold:
\begin{enumerate}
\item $\cVbar_{k,L\oplus L'}(\rho\oplus \sigma) \cong \cVbar_{k,L}(\rho)\oplus \cVbar_{k,L'}(\sigma)$;
\item $\det\cVbar_{k,L}(\rho)\cong \cLbar_{dk- 12\Tr(L)}$;
\item $\cVbar_{k,L}(\rho)^\vee \cong \cVbar_{-k,-L^t}(\rho^\vee)$.
\end{enumerate}
\end{prop}
\begin{proof}
The first claim is obvious. For the second claim recall that if $\cV$ is a vector bundle of rank $r$ and $\cU$ is a vector bundle of rank $t$, then $\det (\cV\otimes \cU) \cong (\det \cV)^{\otimes t}\otimes (\det \cU)^{\otimes r}$, so that $\det\cVbar_{k,L}(\rho)\cong (\det \cVbar_{0,L}(\rho))\otimes \cLbar_{dk}$. Next, since for a square matrix $M$ one has $\det e^{M} = e^{\Tr(M)}$, one sees that $\det \cVbar_{0,L}(\rho)$ is the extension $\cVbar_{0,\Tr(L)}(\det\rho)$. Thus, (2) of the Proposition now follows by Theorem \ref{thm:characters}. The third claim is also obvious from the definition of $\cVbar_{k,L}(\rho)$, since dualizing corresponds to taking inverses and transposing, so that the matrix used to construct $\cVbar_{k,L}(\rho)^\vee$ is $e^{-2\pi i L^t\tau}$.
\end{proof}

Theorem \ref{thm:characters} highlights the fact that the extensions $\cVbar_{k,L}(\rho)$ depend fundamentally on the choice of exponent matrix $L$. Thankfully, the following result of Gantmacher classifies all possible exponent matrices. 

\begin{thm}[Gantmacher \cite{Gantmacher}]
\label{t:gantmacher}
Fix a branch $\log$ of the complex logarithm. Let $G \in \GL_n(\CC)$ have the Jordan canonical form
\[
  Z^{-1}GZ = J = \diag(J_1(\lambda_1),J_2(\lambda_2),\ldots, J_r(\lambda_r)).
\]
Then all solutions to $e^{X} = G$ are given by
\[
  X = ZU\diag(L_1^{(t_1)},L_2^{(t_2)},\ldots,L_r^{(t_r)})U^{-1}Z^{-1},
\]
where, if $J_k(\lambda_k)$ is an $n_k\times n_k$ Jordan block with $\lambda_k$ on the diagonal, then
\[
  L_k^{(t_k)} = 
\left(\begin{smallmatrix}
\log(\lambda_k) + 2\pi i t_k & \lambda_k^{-1} & -\lambda_k^{-2} & \cdots & (-1)^{n_k}\lambda_k^{1-n_k}\\
0 & \log(\lambda_k) + 2\pi i t_k & \lambda_k^{-1} & \cdots & (-1)^{n_k-1}\lambda_k^{2-n_k}\\
0 &0& \log(\lambda_k) + 2\pi i t_k & \cdots & (-1)^{n_k-2}\lambda_k^{3-n_k}\\
&\vdots&&&\vdots\\
0&0&0&\cdots&\log(\lambda_k) +2\pi i t_k 
\end{smallmatrix}\right),
\]
the $t_k$ are arbitrary integers, and $U$ is any invertible matrix that commutes with $J$.
\end{thm}

In particular, if the exponent matrix for a representation $\rho$ is defined by $\rho(T) = e^{2\pi i L}$, then all the eigenvalues of $L$ will be of the form $\frac{1}{2\pi i}\log(\lambda_k)+ t_k$, and by choosing the $t_k$'s appropriately we can arrange for all these eigenvalues to have real part in a given half-open interval of length 1.  

\begin{dfn}
\label{d:standardexponents}
Let $\rho \colon \SL_2(\ZZ) \to \GL(V)$ denote a representation, and let $I \subseteq \RR$ denote a half-open interval of length $1$. Then a \emph{choice of exponents for $\rho$ relative to $I$} is an endomorphism $L$ of $V$ satisfying the two properties:
\begin{enumerate}
\item $\rho(T) = e^{2\pi i L}$;
\item the eigenvalues of $L$ have real part in $I$.
\end{enumerate}
A \emph{standard choice of exponents} for $\rho$ is a choice of exponents relative to $I = [0,1)$.
\end{dfn}

The choice of exponents completely determines the isomorphism class of the extended vector bundle $\cVbar_{k,L}(\rho)$, in the following sense:

\begin{prop}
Let $L_1$ and $L_2$ be two choices of exponents for $\rho$ made relative to the same interval. Then there is an isomorphism $\cVbar_{k,L_1}(\rho)\cong\cVbar_{k,L_2}(\rho)$ depending only on the matrix $U$ of Theorem \ref{t:gantmacher}. 
\end{prop}

\begin{proof}
Decompose $\rho = \rho^+\oplus \rho^-$ into even and odd parts. Then since $\cVbar_{k,L}(\rho^+\oplus\rho^-)\cong \cVbar_{k,L}(\rho^+)\oplus\cVbar_{k,L}(\rho^-)$, we may assume that $\rho(-I_2) = \pm 1$. Under this hypothesis, let $(\cW_1,\cW_2,\phi_{L_1})$ and $(\cU_1,\cU_2,\phi_{L_2})$ be the triples defining $\cVbar_{k,L_1}(\rho)$ and $\cVbar_{k,L_2}(\rho)$, respectively, as in Proposition \ref{prop:extensionOfVValuedModForms}. Since both vector bundles are extensions of  $\cV_k(\rho)$, we have $\cW_i=\cU_i$ for $i=1,2$, and it thus suffices to show that $\phi_{L_1}$ and $\phi_{L_2}$ differ by bundle automorphisms. Assume that $\rho(T)=J$ is in Jordan canonical form, so that $Z=I_n$ in Theorem \ref{t:gantmacher}. By the hypothesis on the choice of interval $I$, the matrices $L_k^{(t_k)}$ in Theorem \ref{t:gantmacher} are the same for $L_1$ and $L_2$, thus we may further assume that $L_1 = UL_2U^{-1}$, where  $U$ is chosen as in Theorem \ref{t:gantmacher}. We then have
\begin{align*}
\phi_{L_1}(v,\tau) &= (e^{-2\pi i\tau UL_2U^{-1}}\,v,q) \\
&= (U\,e^{-2\pi i\tau L_2}U^{-1}\,v,q) \\
&= \phi_U\,\phi_{L_2}\,\phi'_{U^{-1}}(v,\tau)
\end{align*}
where $\phi'_{U^{-1}}(v,\tau) = (U^{-1}v,\tau)$ and $\phi_{U}(v,q) = (Uv,q)$. Now, $\phi'_{U^{-1}}(v,\tau)$ is a bundle automorphism of $\iota_1^*\cV_k(\rho)$, since $U$ commutes with $\rho(T)=J$ by Theorem \ref{t:gantmacher}. On the other hand, $\phi_{U}(v,q) = (Uv,q)$ is a bundle automorphism of $\iota_2^*\cW_2$, since by the assumptions made at the beginning of the paragraph we have that $\rho(-I_2) = \pm 1$, which commutes with $U$ as well. 
\end{proof}

\begin{ex}
Let $\rho = \mathbf{1}$ as in Example \ref{ex:modFormsVsCuspForms}. Then the standard choice of exponents $I = [0,1)$ yields $\cVbar_k(\rho) \cong\cLbar_k$, the line bundle of  weight $k$ holomorphic modular forms, whereas the choice $I = (0,1]$ gives the line bundle of weight $k$ cusp forms. 
\end{ex}

The previous example motivates the following important definition:

\begin{dfn}
Let $\rho \colon \SL_2(\ZZ) \to \GL(V)$ denote a representation, and let $k$ denote an integer. {\em Holomorphic vector valued modular forms} for $\rho$ of weight $k$ are global sections of the extension of $\cV_k(\rho)$ corresponding to a standard choice of exponents for $\rho$. The \emph{holomorphic cusp forms} for $\rho$ of weight $k$ are global sections of the extension of $\cV_k(\rho)$ corresponding to a choice of exponents made relative to the interval $(0,1]$. We denote by $M_k(\rho)$ and $S_k(\rho)$ the vector spaces of weight $k$ holomorphic vector valued modular forms and cusp forms, respectively, for $\rho$.
\end{dfn}

In what follows, the simplified notation $\cVbar_k(\rho)$ denotes the extension of $\cV_k(\rho)$ relative to the standard choice of exponents. Similarly, $\cSbar_k(\rho)$ will always denote the extension of $\cV_k(\rho)$ made relative to the interval $(0,1]$. There is an inclusion $\cSbar_k(\rho) \to \cVbar_k(\rho)$ of vector bundles that is an isomorphism away from $\infty$. However, if $\rho(T)$ does not have $1$ as an eigenvalue, then in fact $\cVbar_k(\rho) = \cSbar_k(\rho)$. 

\begin{rmk}
Multiplication by $\eta^{2n}$, where $n \in \ZZ_{\geq 1}$, defines an injection of vector bundles $\cVbar_k(\rho) \hookrightarrow \cSbar_{k+n}(\rho\otimes \chi^n)$. The image is the bundle obtained by extending $\cV_{k+n}(\rho\otimes \chi^{n})$ using the interval $[\frac{n}{12},\frac{n}{12}+1)$. This is sometimes useful for determining dimensions of spaces of modular forms of weight one.
\end{rmk}

Modular forms and cusp forms are related via duality as follows.

\begin{prop}
\label{p:duality}
For every integer $a$ one has
\[\cVbar_{k,L} (\rho)^\vee \cong \begin{cases}
\cSbar_{a+12-k}(\rho^\vee\otimes \chi^a)& L \text{ relative to } [\frac{a}{12},\frac{a}{12}+1),\\
\cVbar_{a+12-k}(\rho^\vee\otimes \chi^a)& L \text{ relative to } (\frac{a}{12},\frac{a}{12}+1].
\end{cases}
\]
\end{prop}
\begin{proof}
This follows by part (3) of Proposition \ref{p:functorialproperties}, and by multiplying by $\eta^{2a+24}$.
\end{proof}

The following result, and its proof, are due to Geoff Mason \cite{Mason1}. We include the proof as \cite{Mason1} states the result in a slightly weaker form, although Mason's proof in fact gives the following stronger result.
\begin{prop}[Corollary 3.8 of \cite{Mason1}]
\label{c:weightbound}
Let $\rho$ denote a $d$-dimensional representation of $\SL_2(\ZZ)$, and let $L$ denote a choice of exponents for $\rho$. If $\cVbar_{k,L}(\rho)$ has a global section whose component functions are linearly independent over $\CC$, then
\[k\geq \frac{12\Tr(L)}{d} + 1-d.\]
\end{prop}
\begin{proof}
For ease of notation, let $\cVbar_k:= \cVbar_{k,L}(\rho)$. Recall the modular derivative $D_k = q\frac{d}{dq} - \frac{kE_2}{12}$, which maps sections of $\cVbar_k$ to $\cVbar_{k+2}$. Define $D_k^r = D_{k+2(r-1)} \circ \cdots \circ D_{k+2}\circ D_k$. Then if $F = (f_j)$ is a global section of $\cVbar_k$, the so-called modular Wronskian of $F$, as introduced by Mason in \cite{Mason1}, is defined as the determinant
\[
  W(F) = \det \left(\begin{matrix}
f_1 & D_kf_1 & D_k^2 f_1 & \cdots & D^{d-1}_kf_1\\
f_2 & D_kf_2 & D_k^2 f_2 & \cdots & D^{d-1}_kf_2\\
f_3 & D_kf_3 & D_k^2 f_3 & \cdots & D^{d-1}_kf_3\\
 & \vdots & &&\vdots\\
f_d & D_kf_d & D_k^2 f_d & \cdots & D^{d-1}_kf_d\\    
\end{matrix}
  \right).
\]
Thus, $W(F)$ is a global section of $\det\left(\bigoplus_{r = 0}^{d-1} \cVbar_{k+2r}\right) \cong \bigotimes_{r = 0}^{d-1} \det \cVbar_{k+2r}$ by definition, and it is nonzero by hypothesis. Since $\cLbar_x$ only has nonzero global sections if $x > 0$, the claim follows by (2) of Proposition \ref{p:functorialproperties}.
\end{proof}

\begin{rmk}
If $\rho$ is irreducible, then the linear independence hypothesis of Proposition \ref{c:weightbound} is satisfied by any nonzero global section of $\cVbar_{k,L}(\rho)$. In this case Proposition \ref{c:weightbound} gives a lower bound on the minimal weights $k_1$ and $k_2$ such that $M_{k_1}(\rho) \neq 0$ and $S_{k_2}(\rho) \neq 0$.
\end{rmk}

As is well-known, the global sections of $\cLbar_0$ are just the constant functions and the line bundles $\cLbar_k$ have no global sections for $k<0$. These two features need not be true for vector valued modular forms. In general, there is a natural injective map $V^{\SL_2(\ZZ)} \hookrightarrow M_0(\rho)$ whose image consists of constant functions. It is known \cite{KnoppMason3}, \cite{Mason1} that there exist nonconstant vector valued modular forms of weight zero for certain representations $\rho$. Moreover, there are nonzero holomorphic vector valued modular forms of negative weight for certain representations $\rho$ (e.g. $f(\tau) = (\tau,1) \in M_{-1}(\rho)$, where $\rho$ is the standard inclusion $\SL_2(\ZZ) \hookrightarrow \GL_2(\CC)$).
\begin{dfn}
A representation $\rho$ of $\SL_2(\ZZ)$ is said to be \emph{good} if $M_0(\rho\otimes \chi^{a})$ consists only of constant functions for $a = 0,\ldots, 11$.
\end{dfn}
\begin{rmk}
\label{rmk:cuspFormsAreZeroForGoodReps}
Note that if $\rho$ is good then $S_0(\rho \otimes \chi^a) = 0$ for all $a$.
\end{rmk}
\begin{lem}
\label{l:finiteimagegood}
All representations of $\SL_2(\ZZ)$ of finite image are good.
\end{lem}
\begin{proof}
If $F$ is a modular form of weight $0$ for $\rho$ with finite image, then the coordinates of $F$ are holomorphic scalar valued modular forms of weight zero for the finite index subgroup $\ker \rho$ of $\SL_2(\ZZ)$. They thus define global holomorphic functions on a compact Riemann surface, and thus must be constant.
\end{proof}
\begin{dfn}
\label{dfn:positive}
A representation $\rho$ of $\SL_2(\ZZ)$ is said to be \emph{positive} if $M_k(\rho \otimes \chi^a) = 0$ for all integers $k < 0$ and $a = 0,\ldots, 11$.
\end{dfn}
\begin{lem}
\label{l:weightzero}
A representation $\rho$ is positive if either of the following conditions hold:
\begin{enumerate}
\item $\rho$ is good;
\item $\rho$ is unitarizable.
\end{enumerate}
\end{lem}
\begin{proof}
For (1) note that if $k < 0$, then multiplication by $\eta^{-2k}$ defines an injective map $M_k(\rho) \hookrightarrow S_0(\rho\otimes \chi^{-k})$. But $S_0(\rho \otimes \chi^{-k}) = 0$ if $\rho$ is good. For (2), one can consult the discussion following Lemma 4.1 of \cite{KnoppMason1}.
\end{proof}


\section{Vector bundles over weighted projective lines}
\label{Section:vbundlesWeightedProjectiveLines}

In this section we summarize a few facts about weighted projective lines that will be needed below. The material of this section is entirely independent from the rest of the paper and it is mainly due to Lennart Meier (\cite{Meier}).

Let $n_1,n_2$ be integers. The ring homomorphism
\begin{align*}
\CC[x_1,x_2] &\longrightarrow \CC[t,t^{-1}]\otimes \CC[x_1,x_2] \\
x_i &\longmapsto t^{n_i}\,x_i, \quad i=1,2,
\end{align*}
defines a group-scheme action
$$
\mu: \mathbf{G}_m \times \mathbf{A}^2 \longrightarrow \mathbf{A}^2.
$$
Let $\PP(n_1,n_2)\df \mathbf{A}^{2,\times}_{\CC}\dbls \mathbf{G}_m$ denote the quotient in the category of algebraic stacks of the action $\mu$ restricted to  the open subscheme $\mathbf{A}^{2,\times} = \mathrm{Spec}(\CC[x_1,x_2]) - \{(0,0)\}$. This quotient is called the {\em weighted projective line} with weights $n_1$ and $n_2$. It is a proper smooth algebraic stack. Note that $\PP(1,1) = \PP^1$, the usual projective line. 

A {\em vector bundle} of rank $r$ over  $\PP(n_1,n_2)$ is a $\mathbf{G}_m$-equivariant vector bundle on  $\mathbf{A}^{2,\times}$, that is, a locally free sheaf $\cV$ of rank $r$ over $\mathbf{A}^{2,\times}$ together with an isomorphism
$$
\varphi: \mathrm{pr}^*\cV \stackrel{\cong}\longrightarrow \mu^*\cV,
$$
satisfying the standard cocycle condition, where $\mathrm{pr}: \mathbf{G}_m \times \mathbf{A}^{2,\times} \longrightarrow \mathbf{A}^{2,\times}$ denotes the projection map on the second coordinate.

As is the case for $\PP^1$, the study of vector bundles over $\PP(n_1,n_2)$ is equivalent to the study of finitely generated graded modules. In particular, let $S_{n_1,n_2}$ be the graded $\CC$-algebra given by the polynomial algebra $\CC[x_1,x_2]$ where $x_1$ and $x_2$ are of degree $n_1$ and $n_2$, respectively. A vector bundle over $\PP(n_1,n_2)$ can be extended to a $\mathbf{G}_m$-equivariant coherent sheaf over the affine plane $\mathbf{A}^2_{\CC}$, and since $\mathbf{A}^2_{\CC}$ is affine with coordinate ring $\CC[x_1,x_2]$, this coherent sheaf is equivalent to a finitely generated graded $S_{n_1,n_2}$-module, which we denote by $\cV^{\sim}$. The key point, due to Lennart Meier (who in turn credits Angelo Vistoli), is to observe that  $\cV^{\sim}$ is {\em projective}: 

\begin{thm}[\cite{Meier}, Proof of Prop. 3.4]
\label{thm:equivCategoriesGradedSModules}
The functor
\begin{align*}
\mathrm{Vec}(\PP(n_1,n_2)) &\longrightarrow \mathrm{prgr}(S_{n_1,n_2}) \\
\cV &\longmapsto \cV^{\sim}
\end{align*}
is an equivalence of categories between the category of vector bundles over $\PP(n_1,n_2)$ and the category of projective, finitely generated graded $S_{n_2,n_2}$-modules. 
\end{thm}

Let $M$ be a graded $S_{n_1,n_2}$-module. For any integer $i\geq 0$, let $M[i]$ denote the homogeneous component of degree $i$ in $M$. For any integer $k\in \ZZ$, let  $M(k)$ be the graded $S_{n_1,n_2}$-module given by $M$, but with grading given by $M(k)[i] = M[i+k]$. 

\begin{dfn}
For any $k\in\ZZ$, the line bundle $\cO(k)$ over $\PP(4,6)$ is the unique line bundle such that $\cO(k)^{\sim} = S_{n_1,n_2}(k)$, where $\cV \mapsto \cV^{\sim}$ is the functor of Theorem \ref{thm:equivCategoriesGradedSModules}.
\end{dfn}

Theorem \ref{thm:equivCategoriesGradedSModules} implies:

\begin{thm}[\cite{Meier}, Prop. 3.4]
\label{thm:DecompositionTheorem}
Any vector bundle $\cV$ of rank $n$ over $\PP(n_1,n_2)$ decomposes as $\mathcal{V} \cong \bigoplus_{i=1}^n \cO(a_i)$ for uniquely determined integers $a_1,\ldots,a_n \in \ZZ$ with $a_1 \geq a_2 \geq \cdots \geq a_n$.  
\end{thm}

A useful consequence of Theorem \ref{thm:DecompositionTheorem} is that the cohomology of $\cV$ can be computed in terms of the cohomology of the $\cO(a_i)$'s, which is well-known:

\begin{prop}[\cite{Meier}, \S 2]
\label{prop:cohomologyOfLineBundles}
For any $k\in\ZZ$, we have:
\begin{itemize}
\item[(i)] $H^0( \PP(n_1,n_2), \cO(k)) \cong \bigoplus_{(a,b)\in I_0} \CC\,x_1^a\,x_2^b $, where  $$I_0 = \{(a,b)\in \ZZ_{\geq 0}\times\ZZ_{\geq 0}:an_1 + bn_2 = k\}.$$  
\item[(ii)] $H^1( \PP(n_1,n_2), \cO(k)) \cong \bigoplus_{(c,d)\in I_1} \CC\,x_1^c\,x_2^d $, where  $$I_1 = \{(c,d)\in \ZZ_{< 0}\times\ZZ_{< 0}:cn_1 + dn_2 = k\}.$$  
\item[(iii)] $H^i( \PP(n_1,n_2), \cO(k)) = 0$, for all $i \geq 2$. 
\end{itemize}
\end{prop}

Proposition \ref{prop:cohomologyOfLineBundles} allows one to deduce a relationship between the cohomology of $\cV$ and that of its dual $\cV^{\vee}$, as follows: 
\begin{prop}[Weak Serre Duality]
\label{prop:weakSerreDuality}
Let $\cV$ be a vector bundle over $\PP(n_1,n_2)$. Then 
$$
h^0(\cV) = h^1(\cV^{\vee}\otimes\cO(-n_1-n_2)).
$$
\end{prop}

\begin{proof}
By Theorem \ref{thm:DecompositionTheorem}, we can write $\mathcal{V} \cong \bigoplus_{i=1}^r \cO(a_i)$ and $\mathcal{V}^{\vee} \cong \bigoplus_{i=1}^r \cO(-a_i)$. Now it is clear by Proposition \ref{prop:cohomologyOfLineBundles} that $h^0(\cO(k)) = h^1(\cO(-k-n_1-n_2))$ for any $k\in\ZZ$, and the result thus follows by applying this identity to each component $\cO(a_i)$.
\end{proof}

\begin{rmk}
The expert reader will notice that $\cO(-n_1-n_2)$ is the canonical bundle of $\PP(n_1,n_2)$, as follows for example by the weighted Euler sequence for $\PP(n_1,n_2)$ (\cite{Edidin}, 4.2.1). Therefore Proposition \ref{prop:weakSerreDuality} should just be a manifestation of `Serre duality for weighted projective lines'. However, we could not find a reference in the literature for such statement, and we therefore chose to prove it in this very weak form. 
\end{rmk}

In light of Theorem \ref{thm:equivCategoriesGradedSModules}, the statement of Theorem \ref{thm:DecompositionTheorem} is also equivalent to the following:

\begin{thm}
\label{thm:FreeModuleTheoremWeightedProjectiveStack}
Let $\cV$ be a vector bundle of rank $n$ over $\PP(n_1,n_2)$. Then $\cV^{\sim} \cong \bigoplus_{i=1}^n S_{n_1,n_2}(a_i)$ is a free $S_{n_1,n_2}$-module of rank $n$.
\end{thm}

\section{Roots and the free-module theorem}
The modular orbifold $\cMbar^{\rm an}$ is the analytification of the moduli stack $\cMbar$ of generalized elliptic curves, which is a smooth and proper algebraic stack. Moreover, there is a well-known isomorphism of algebraic stacks (e.g. \cite{Meier}, Example 2.4) $\cMbar \cong \PP(4,6)$. By GAGA for proper algebraic stacks (\cite{Toen}, \S 5.2), the analytification functor
\begin{align*}
\mathrm{Coh}(\cMbar) &\longrightarrow \mathrm{Coh}(\cMbar^{\rm{an}}) \\
\cF &\longmapsto \cF^{\rm an}
\end{align*}
between the corresponding categories of coherent sheaves induces an equivalence of categories, such that
$$
H^i(\cMbar, \cV) = H^i(\cMbar^{\rm{an}}, \cV^{\rm an}).
$$ 
In particular, to each vector bundle $\cVbar_{k,L}(\rho)$ over $\cMbar^{\rm{an}}$ we can associate a vector bundle $\cV$ over $\cMbar \cong \PP(4,6)$, whose analytification is $\cVbar_{k,L}(\rho)$, and with identical cohomology. Since we are only interested in cohomological computations, there is no harm in denoting the (algebraic) vector bundle $\cV$ over $\cMbar$ also by $\cVbar_{k,L}(\rho)$. 

\begin{ex}
\label{ex:modularFormsCorrespondToO(k)}
If $\cVbar_{k,L}(\rho) = \cVbar_{k,0}(\mathbf{1}) = \cLbar_k$, then the corresponding line bundle over $\PP(4,6)$ is just $\cO(k)$. Similarly  $\cSbar_{k}(\mathbf{1})$, the line bundle of weight $k$ cusp forms, corresponds to $\cO(k-12)$. 
\end{ex}

The machinery of Section \ref{Section:vbundlesWeightedProjectiveLines} may thus be applied to the study of the vector bundles $\cVbar_{k,L}(\rho)$ attached to a representation $\rho:\SL_2(\ZZ)\rightarrow \GL(V)$. In particular, Theorem \ref{thm:DecompositionTheorem} gives a decomposition 
\begin{equation}
\label{eqn:rootEquation}
\cVbar_0(\rho) \cong \bigoplus_{i=1}^{d=\dim \rho} \cLbar_{a_i}, 
\end{equation}
for uniquely determined integers $a_1\geq a_2 \geq \ldots \geq a_d$, which depend on the representation $\rho$ only. 

\begin{dfn}
\label{dfn:roots}
The integers $a_1\geq a_2 \geq \ldots \geq a_d$ of \eqref{eqn:rootEquation} are called the {\em roots} of $\rho$. 
\end{dfn} 

The roots of $\rho$ entirely determine the cohomology of $\cVbar_k(\rho)$ for all integers $k$. 

\begin{ex}
If $\rho = \chi^a$ is a character, $a=0,\ldots,11$, then Theorem \ref{thm:characters} gives $\cVbar_{0}(\chi^a) \cong \cO(-a)$, so the only root is $a_1=-a$. In particular,
$$
h^i(\cVbar_k(\chi^a)) = h^i(\cO(k-a)), \quad i=1,2,
$$
so dimension formulas for the spaces $M_k(\chi^a)$ can be read off from Proposition \ref{prop:cohomologyOfLineBundles} with $n_1=4$ and $n_2=6$. 
\end{ex}

Finding the roots of $\rho$ can be harder in higher rank, and the issue will be addressed more properly in Section \ref{s:DimensionFormulae} below. There are however some very general restrictions on the roots which are easy to derive. For example, note that since $\rho(S)$ (resp. $\rho(R)$) is of order 4 (resp. 6), its eigenvalues will be of the form $i^s$ (resp. $\xi^r$, $\xi = e^{2\pi i/6}$) for $s=0,\ldots,3$ (resp. $r=0,\ldots,5$). The multiplicities of these eigenvalues give restrictions on the roots of $\rho$ as follows:

\begin{thm}
\label{thm:rootCongruences}
 Let $\alpha_s$ (resp. $\beta_r$) be the multiplicity of the eigenvalue $i^s$ (resp. $\xi^r$) of $\rho(S)$ (resp. $\rho(R)$). Then precisely $\alpha_s$ roots of $\rho$ are congruent to $s$ modulo 4 and precisely $\beta_r$ roots of $\rho$ are congruent to $r$ modulo 6.
\end{thm} 

\begin{proof}
The point $i\in \frakh$ descends to a geometric point $\kappa(i): \Spec(\CC) \rightarrow \cMbar$ whose stabilizer is cyclic of order 4, generated by $S$. The vector bundle $\kappa(i)^*\cVbar_{L,0}(\rho)$ is just a copy of the vector space $V$ together with the action of the cyclic group generated by $\rho(S)$. Now the root decomposition gives an isomorphism $\cVbar_{0}(\rho) \simeq \bigoplus_{j=1}^d \cLbar_{a_j}$, and since the action of $S$ on $\kappa(i)^*\cLbar_{a_j}$ is given by $i^{a_j}$, we deduce that the action of $\rho(S)$ on $V$ can be diagonalized as $\rho(S)\sim \diag(i^{a_1},\ldots,i^{a_d})$. The same reasoning applies to the geometric point $\kappa(\zeta): \Spec(\CC) \rightarrow \cMbar$, where $\zeta = e^{2\pi i/3}$, whose stabilizer is cyclic of order 6, generated by $R$.      
\end{proof}

Another consequence of viewing $\cVbar_{k,L}(\rho)$ as vector bundles over $\PP(4,6)$ is that Theorem  \ref{thm:FreeModuleTheoremWeightedProjectiveStack}, applied to $n_1=4,n_2=6$ and $\cV = \cVbar_0(\rho)$, implies the well-known \emph{free-module theorem} for vector valued modular forms. In particular, the statement below generalizes (in the case of integral weights) that of \cite{Gannon}, Theorem 3.4, to arbitrary representations $\rho$.  

\begin{thm}
\label{thm:FreeModuleTheorem}
Let $\rho:\SL_2(\ZZ)\rightarrow \GL_n(\CC)$ be a representation and let
\[
  M(\rho) \df \bigoplus_{k \in \ZZ} H^0\left(\cMbar,\cVbar_k(\rho)\right)
\]
denote the corresponding module of holomorphic vector valued modular forms for $\rho$. Then
\begin{itemize}
\item[(i)] $M(\rho)$ is a free module of rank $n$ over $M(1)$, the ring of holomorphic scalar-valued modular forms of level one.
\item[(ii)] Let $k_1\leq \ldots \leq k_n$ be the weights of the free generators. Then, using the notation of Theorem \ref{thm:rootCongruences}, precisely $\alpha_s$ (resp. $\beta_r$) of these weights are congruent to $-s$ mod 4 (resp. $-r$ mod 6). Moreover, $\sum_j k_j = 12\Tr(L)$.   
\end{itemize}
\end{thm}

\begin{proof}
By Theorem \ref{thm:FreeModuleTheoremWeightedProjectiveStack} with $n_1=4$ and $n_2=6$ we know that $\cVbar_0(\rho)^{\sim} \cong \bigoplus_{i=1}^n S_{4,6}(a_i)$ is free of rank $n$ over $S_{4,6} \cong M(1)$. Now the line bundle $\cLbar_k$ over $\cMbar^{\rm an}$ corresponds to the line bundle $\cO(k)$ over $\PP(4,6)$, as in Example \ref{ex:modularFormsCorrespondToO(k)}. Thus for all $a_i$ we have
$$
S_{4,6}(a_i) \cong \bigoplus_{k \in \ZZ} H^0\left(\PP(4,6),\cO(k+a_i)\right) \cong \bigoplus_{k \in \ZZ} H^0\left(\cMbar,\cLbar_{k+a_i}\right)
$$
and therefore $\cVbar_{0,L}(\rho)^{\sim} \cong M(\rho)$, which proves part (i). To prove (ii) it suffices to note that $k_j = -a_j$ and then apply Theorem \ref{thm:rootCongruences} and Proposition \ref{p:functorialproperties} part (2).
\end{proof}

\begin{rmk}
The same proof shows that a corresponding free-module theorem is also true for every vector bundle $\cVbar_{0,L}(\rho)$, not just the one obtained from a standard choice of exponents. In particular, the graded module of holomorphic cusp forms is also free of rank $\dim \rho$.
\end{rmk}
\begin{rmk}
In \cite{Gannon}, \S 3.4, Gannon points out that the free-module theorem is proved in \cite{EichlerZagier}, although it is not stated as above.
\end{rmk}

\section{Dimension formulae}
\label{s:DimensionFormulae}

The Riemann-Roch Theorem for weighted projective lines (\cite{Edidin}, 4.2.5) allows one to compute the Euler characteristics of the vector bundles $\cVbar_{k,L}(\rho)$. In many cases, this is enough to obtain a dimension formula for these spaces of vector valued modular forms. Whenever a dimension formula is available, one can use it to compute the roots of $\rho$, in the sense of Definition \ref{dfn:roots}. This section explains these computations, and then several examples are illustrated in Section \ref{s:examples}.

To state the relevant formulas for the Euler characteristic, again consider the weighted projective line $\PP(4,6)$, and for simplicity let $X\df\mathbf{A}^{2,\times}_{\CC}$. For $h\in \mathbf{G}_m$, we can consider the locus $X^h$ of points that are fixed by $h$. In particular, we have
\[X^h =\begin{cases}
 X &h=\pm 1, \\
\{(x,y)\in X: y=0 \}\cong \CC^{\times} &h=\pm i, \\
\{(x,y)\in X: x=0 \}\cong \CC^{\times} &h=\zeta^{\pm 1}\textrm{ or } h= \xi^{\pm 1} \\
\emptyset &\text{otherwise},
\end{cases}\]
where $\zeta = e^{2\pi i/3}$ and $\xi= e^{2\pi i/6}$. The action of $\mathbf{G}_m$ restricts to $X^h$, and for each $h$ we may take the corresponding quotient in the category of stacks:
\[
X^h\dbls\mathbf{G}_m \cong \begin{cases} 
\PP(4,6) & h=\pm 1, \\
B\mu_4  &h=\pm i, \\
B\mu_6 &h=\zeta^{\pm 1} \textrm{ or } h= \xi^{\pm 1},\\
\emptyset &\text{otherwise},
\end{cases}
\]
where by $B\mu_n$ we have denoted the stack quotient $\mathbf{G}_m\dbls\mathbf{G}_m$ by the action $\lambda\mapsto \lambda^n$, the classifying stack of $\mu_n$-torsors over $\Spec(\CC)$. For any $h$, consider the embedding
$$
\iota_h: X^h\dbls\mathbf{G}_m \hookrightarrow \PP(4,6).
$$ 
If $\mathcal{V}$ is a vector bundle of rank $n$ over $\PP(4,6)$, then $\iota_h^*\mathcal{V}$ is a vector bundle on the stack $X^h \dbls \mathbf{G}_m$. In particular, for $h = \pm i, \zeta^{\pm 1}, \xi^{\pm 1}$, the vector bundle $\iota_h^*\mathcal{V}$ is just a $n$-dimensional $\CC$-vector space together with an action of a linear operator $h|_{\cV}$, of order 4, 3 or 6, respectively. On the other hand, for $h=-1$ the vector bundle $\iota^*_h\cV$ is canonically isomorphic to $\cV$. The action of $h=-1$ thus gives a bundle automorphism of order 2, and we may write $\mathcal{V} \cong \mathcal{V}^{+}\oplus\mathcal{V}^{-}$ for the decomposition into eigenbundles. Finally, for each vector bundle $\cV$ over $\PP(4,6)$ let $d(\cV)$ denote the unique integer such that $\det(\cV) \cong \cO(d(\cV))$, which is well-defined since $\Pic(\PP(4,6))\cong \ZZ$, generated by $\cO(1)$. The following formula for $\chi(\cV)$ follows directly from the much more general Riemann-Roch Theorem of \cite{Edidin} (in particular, see Exercise 4.11 of \cite{Edidin}). 

\begin{thm}[\cite{Edidin},Theorem 4.10]
\label{thm:eulerFormula}
Let $\mathcal{V}$ be a vector bundle over $\PP(4,6)$. Then
\begin{equation*}
\begin{aligned}
\chi( \PP(4,6),\mathcal{V}) &= \frac{1}{24}\left(5\rk(\cV) + d(\cV)\right) + \frac{1}{24}\left(5\rk(\cV^+) - 5\rk(\cV^-) + d(\mathcal{V}^{+}) - d(\mathcal{V}^{-})\right) \\
&+ \frac{1}{8}\Tr(i|_{\mathcal{V}}) + \frac{1}{8}\Tr(-i|_{\mathcal{V}}) + \frac{1}{6(1-\zeta^{-1})}\Tr(\zeta|_{\mathcal{V}}) + \frac{1}{6(1-\zeta)}\Tr(\zeta^{-1}|_{\mathcal{V}})\\
&+ \frac{1}{6(1-\zeta)}\Tr(\xi|_{\mathcal{V}})+ \frac{1}{6(1-\zeta^{-1})}\Tr(\xi^{-1}|_{\mathcal{V}}).
\end{aligned}
\end{equation*}
\end{thm}
Applying Theorem \ref{thm:eulerFormula} to the vector bundles $\cVbar_{k,L}(\rho)$ yields the following result.
\begin{cor}
\label{c:eulerchar}
Let $\rho\colon \SL_2(\ZZ) \to \GL(V)$ denote an $n$-dimensional representation of the form $\rho = \rho^+\oplus \rho^-$ where $\rho^+$ is even and $\rho^-$ is odd, let $L = L^+\oplus L^-$ denote a choice of exponents for $\rho$ adapted to the decomposition $\rho = \rho^+\oplus  \rho^-$, and let $\cVbar_{k,L}(\rho)$ denote the corresponding bundle of weight $k$ modular forms for $\rho$. Then
\begin{align*}
\chi(\cVbar_{k,L}(\rho)) &= \begin{cases}
\frac{(5+k)\dim\rho^+}{12} + \frac{i^k\Tr(\rho^+(S))}{4} + \frac{\xi^k\Tr(\rho^+(R))}{3(1-\zeta)} + \frac{\zeta^k\Tr(\rho^+(R^{2}))}{3(1-\zeta^{-1})} - \Tr(L^+)  & \textrm{if } 2\mid k,\\
\frac{(5+k)\dim\rho^-}{12}  + \frac{i^k\Tr(\rho^-(S))}{4}+ \frac{\xi^k\Tr(\rho^-(R))}{3(1-\zeta)} + \frac{\zeta^k\Tr(\rho^-(R^{2}))}{3(1-\zeta^{-1})} - \Tr(L^-)  &\textrm{if } 2\nmid k.
\end{cases}
\end{align*}
\end{cor}

\begin{proof}
We have 
\begin{align*}
\det(\cVbar_{k,L}(\rho)) &= \cO(-12\Tr(L) + kn), \\
\det(\cVbar_{k,L}^{\pm 1}(\rho)) &= \cO(-12\Tr(L^{\pm 1}) + k\,\rk(\cVbar^{\pm 1}_{k,L}(\rho))),
\end{align*}
by Theorem \ref{thm:characters} and 
$
\rk(\cVbar^+_{k,L}(\rho)) - \rk(\cVbar^-_{k,L}(\rho)) = (-1)^k\Tr(\rho(-I_2)).
$
Moreover the linear maps $h|_{\cVbar_{k,L}(\rho)}$, for $h = \pm i, \zeta^{\pm 1}$ and $\xi^{\pm 1}$ correspond to the matrices $\rho(S)^{\pm 1}$, $\rho(R^2)^{\pm 1}$ and $\rho(R)^{\pm 1}$ of orders 4, 3 and 6, respectively. Thus, specializing Theorem \ref{thm:eulerFormula} to the vector bundles $\cVbar_{k,L}(\rho)$ yields
\begin{align*}
\chi(\cVbar_{k,L}(\rho)) &= \frac{n(5+k)}{24} - \frac{1}{2}\Tr(L) + (-1)^k\left(\frac{5+k}{24}\Tr(\rho(-I_2)) - \frac{\Tr(L^{+}) - \Tr(L^{-})}{2}\right) \\
&+ \frac{i^k}{8}\Tr(\rho(S)) + \frac{i^{-k}}{8}\Tr(\rho(S^{-1})) + \frac{\zeta^k}{6(1-\zeta^{-1})}\Tr(\rho(R^2)) + \frac{\zeta^{-k}}{6(1-\zeta)}\Tr(\rho(R^{-2}))\\
&+ \frac{\xi^k}{6(1-\zeta)}\Tr(\rho(R))+ \frac{\xi^{-k}}{6(1-\zeta^{-1})}\Tr(\rho(R^{-1})).
\end{align*}
It is then elementary to deduce the desired formula.
\end{proof}

The Euler characteristic computation of Corollary \ref{c:eulerchar} yields a dimension formula for positive representations (Definition \ref{dfn:positive})  as follows.
\begin{thm}[Dimension formula]
\label{t:dimension}
Let $(V,\rho)$ denote a positive representation of $\SL_2(\ZZ)$. Then
\[
\dim M_k(\rho) = \begin{cases}
\chi(\cVbar_1(\rho)) + \dim S_1(\rho^\vee)&k = 1,\\
\chi(\cVbar_k(\rho))&k \geq 2,  
\end{cases}
\]
and
\[ 
\dim S_k(\rho) = \begin{cases}
\chi(\cSbar_1(\rho)) + \dim M_1(\rho^\vee)&k = 1,\\
\chi(\cSbar_2(\rho)) + \dim (V^\vee)^{\SL_2(\ZZ)}&k= 2,\\
\chi(\cSbar_k(\rho)) & k \geq 3.  
\end{cases}
\]
If $\rho$ is in fact good, then $\dim M_0(\rho) = \dim V^{\SL_2(\ZZ)}$ and $S_0(\rho) = 0$.
\end{thm}
\begin{proof}
Weak Serre-duality for $\PP(4,6)$ (Proposition \ref{prop:weakSerreDuality}) and Proposition \ref{p:duality} with $a=0$ together yield $h^1(\cVbar_{k}(\rho))= \dim S_{2-k}(\rho^{\vee})$ and $h^1(\cSbar_{k}(\rho)) = \dim M_{2-k}(\rho^{\vee})$, which is true for any representation $\rho$. When  $\rho$ is good, the formula follows from Remark \ref{rmk:cuspFormsAreZeroForGoodReps}.
\end{proof}
\begin{rmk}
If a $d$-dimensional representation $\rho$ is not necessarily positive, the identity $h^1(\cVbar_{k}(\rho))= \dim S_{2-k}(\rho^{\vee})$, combined with the bound of Proposition \ref{c:weightbound} applied to $\rho^\vee$, together imply that $\dim M_k(\rho) = \chi(\cVbar_k(\rho))$ whenever $k > d+1+\frac{12\Tr(L)}{d}$, where $L$ denotes a standard choice of exponents for $\rho(T)$.
\end{rmk}

If $\rho$ is even then Theorem \ref{t:dimension} gives a simple dimension formula for $M_k(\rho)$ and $S_k(\rho)$ in all weights. If $\rho$ is odd then Theorem \ref{t:dimension} does not give a formula for either $\dim M_1(\rho)$ or $\dim S_1(\rho)$. Section \ref{s:examples} contains several examples where positivity determines $\dim M_1(\rho)$ uniquely. More generally, one can map $M_1(\rho)$ into $S_2(\rho \otimes \chi)$ via multiplication by $\eta^2$. It is then often possible to compute $S_2(\rho \otimes \chi)$ and determine which forms are divisible by $\eta^2$. For example, if no standard exponent of $\rho \otimes \chi$ lies in $[0,1/12)$, then $M_1(\rho) \cong S_2(\rho \otimes \chi)$.

One can derive explicit formulae for the roots of $\cVbar_k(\rho)$. To this end we introduce the generating function $P(X) = \sum_{k \in \ZZ} \dim M_k(\rho)X^k$. If $\cVbar_k \cong \bigoplus_{j = 1}^{d} \cObar(k-k_j)$, then we must also have 
\[P(X) = \frac{X^{k_1} + \cdots + X^{k_d}}{(1-X^4)(1-X^6)}.\] 
Thus, by computing $P(X)$ using Theorem \ref{t:dimension}, we may deduce the decomposition of $\cVbar_k(\rho)$ into line bundles. Order the integers $k_j$ so that $k_{j} \leq k_{j+1}$ for all $j$, so that $k_1$ is the \emph{minimal weight} of $\rho$. By Proposition \ref{c:weightbound} we have $k_1 \geq 1-d + \Tr(L)/d$ for a standard choice of exponents $L$.

Assume that $\rho$ is positive, so that Theorem \ref{t:dimension} holds, and $k_1 \geq 0$. Decompose $\rho \cong \rho^+\oplus \rho^-$ into even and odd parts, let $x = \dim M_0(\rho)$, and let $y = \dim S_1(\rho^\vee)$.  Set $d^\pm = \dim \rho^\pm$, $s^\pm = \Tr(\rho^\pm(S))$, $r_1^\pm = \Tr(\rho^\pm(R))$ and $r_2^\pm = \Tr(\rho^\pm(R^2))$. Then the even weight multiplicities are as follows:
\[
\begin{array}{r|l}
\textrm{Weights} & \textrm{Multiplicities}\\
\hline
0&x\\
&\\
2&\frac{7}{12}d^+ - \frac{1}{4}s^++\frac{(\zeta-1)}{9}r_1^+-\frac{(\zeta+2)}{9}r^+_2-\Tr(L^+)\\
&\\
4&\frac{3}{4}d^+ + \frac{1}{4}s^+-\frac{(2\zeta+1)}{9}r_1^++\frac{(2\zeta+1)}{9}r_2^+-x-\Tr(L^+)\\
&\\
6&\frac{1}{3}d^++\frac 13 r_1^++\frac 13 r_2^+-x\\
&\\
8&-\frac{1}{4}d^++ \frac{1}{4}s^++\frac{(2\zeta+1)}{9}r_1^+ -\frac{(2\zeta+1)}{9}r_2^++\Tr(L^+)\\
&\\
10&-\frac{5}{12}d^+- \frac{1}{4}s^+-\frac{(\zeta+2)}{9}  r_1^+ + \frac{(\zeta-1)}{9}  r_2^++x+\Tr(L^+)\\
\end{array}
\]
The odd weight multiplicites are as follows:
\[
\begin{array}{r|l}
\textrm{Weights} & \textrm{Multiplicities}\\
\hline
1& \frac{1}{2}d^-+\frac{i}{4} s^- + \frac{(2\zeta+1)}{9}r_1^-+\frac{(2\zeta+1)}{9}r_2^-+y-\Tr(L^-)\\
&\\
3&\frac{2}{3}d^--\frac i4 s^- - \frac{(\zeta+2)}{9}r_1^- - \frac{(\zeta-1)}{9}r_2^--\Tr(L^-)\\
&\\
5&\frac{1}{3}d^- -\frac{\zeta}{3}r_1^- - \frac{(\zeta+1)}{3}r_2^--y\\
&\\
7&- \frac{1}{6}d^- - \frac i4 s^- + \frac{(\zeta+2)}{9}r_1^-+\frac{(\zeta-1)}{9}r_2^--y+ \Tr(L^-)\\
&\\
9&-\frac{1}{3}d^-+\frac i4 s^- + \frac{(\zeta-1)}{9}r_1^- + \frac{(\zeta+2)}{9}r_2^-+\Tr(L^-)\\
&\\
11&y
\end{array}
\]
The roots of $\cVbar_0(\rho)$ are the negatives of these weights. In particular, the roots of a positive representation always lie between $0$ and $-11$. This was observed by Bantay in \cite{Bantay}.
\begin{rmk}
The table above should agree with Tables III and IV of \cite{Bantay}. The formulae of \cite{Bantay} are defined relative to a choice of exponents for $\rho$ which makes a certain principal part map, discussed in \cite{Borcherds} and \cite{BantayGannon}, bijective. Such a choice always exist, as is proved in \cite{BantayGannon}, and in practice one can compute such an exponent matrix. It does not appear that an explicit formula for the exponent matrix figuring in \cite{BantayGannon} and \cite{Bantay} is known, however.
\end{rmk}

The restrictions on the roots above has the following consequence for scalar valued modular forms.
\begin{prop}
\label{p:generators}
Let $\Gamma \subseteq \SL_2(\ZZ)$ denote a subgroup of finite index $n$, and let $M(\Gamma)$ denote the ring of holomorphic scalar modular forms for $\Gamma$. Then there exists a finite number of modular forms $f_i \in M_{k_i}(\Gamma)$ for $i = 1,\ldots, n-1$ of weights $k_i$ satisfying $1 \leq k_i \leq 11$, such that
\[
  M(\Gamma) = \CC[E_4,E_6] \oplus \bigoplus_{i =1}^{n-1} \CC[E_4,E_6]f_i.
\]
\end{prop}
\begin{proof}
Let $\rho$ denote the permutation representation of $\SL_2(\ZZ)$ on the cosets of $\Gamma$ in $\SL_2(\ZZ)$, which is a good representation. Choose a basis for $\rho$ such that the first basis element is the trivial coset $\Gamma$. Since $M(\rho)$ is a free $\CC[E_4,E_6]$-module with $n$ generators in weights $0 \leq k \leq 11$, and since projection to the first coordinate yields an isomoprhism $M(\rho) \cong M(\Gamma)$, one can take for the $f_i$ the first coordinates of these generators.
\end{proof}

\section{Examples}
\label{s:examples}
\begin{ex}
There is a unique normal subgroup $\Gamma_n$ of $\SL_2(\ZZ)$ with cyclic quotient of order $n$ for each $n \mid 12$. The corresponding decomposition of the ring of modular forms as in Proposition \ref{p:generators} is $M(\Gamma_n) = \bigoplus_{i = 0}^{n-1}\CC[E_4,E_6]\eta^{\frac{24 i}{n}}$. This example shows that the weight bounds in Proposition \ref{p:generators} are sharp.
\end{ex}
\begin{ex}
Consider $\Gamma(2) \subseteq \SL_2(\ZZ)$, which is a normal subgroup with quotient isomorphic with $S_3$. The permutation representation $\rho$ of the cosets is thus the regular representation of $S_3$, and so $\rho \cong 1 \oplus \chi^6 \oplus 2\phi$, where $\phi$ is the $2$-dimensional irreducible of $S_3$. One can use the results of \cite{FrancMason} to make the decomposition of Proposition \ref{p:generators} quite explicit. To explain this, note that $T^2\in \Gamma(2)$, and thus $T$ maps to a two-cycle in $S_3$. It follows that the exponents of $\phi(T)$ are $0$ and $\frac 12$. Thus, Section 4.1 of \cite{FrancMason} tells us that if we write
\begin{align*}
f_1 &= \eta^4 \left(\frac{1728}{j}\right)^{-\frac{1}{6}}\ _2F_1\left(-\frac 16, \frac{1}{6}; \frac 12; \frac{1728}{j}\right), & f_2 &= \eta^4 \left(\frac{1728}{j}\right)^{\frac{1}{3}}\ _2F_1\left(\frac 13, \frac{2}{3}; \frac 32; \frac{1728}{j}\right)
\end{align*}
then
\[
  M(\Gamma(2)) = \CC[E_4,E_6] \oplus \CC[E_4,E_6]\eta^{12} \oplus \bigoplus_{i = 1}^2\bigoplus_{j = 0}^1\CC[E_4,E_6]D^jf_i,
\]
where $D = q\frac{d}{dq} - \frac{E_2}{6}$ is the modular derivative in weight $2$.

There is another well-known description of $M(\Gamma(2))$: the Weierstrass form of a complex analytic elliptic curve $\CC/\Lambda_\tau$, where $\Lambda_\tau = \ZZ \oplus \ZZ\tau$ for $\tau \in \uhp$, is 
\[
y^2 = 4x^3 - g_4(\tau)x-g_3(\tau) = 4(x-e_1(\tau))(x-e_2(\tau))(x-e_3(\tau)),
\] 
where $e_1,e_2,e_3 \in M_2(\Gamma(2))$ are the functions
\begin{align*}
e_1(\tau) &= \wp_{\Lambda_\tau}\left(\frac 12\right),& e_2(\tau) &= \wp_{\Lambda_\tau}\left(\frac{\tau}{2}\right),& e_3(\tau) &= \wp_{\Lambda_\tau}\left(\frac{\tau+1}{2}\right),
\end{align*}
and where $\wp_\Lambda(z)$ is the Weierstrass $\wp$-function of a lattice $\Lambda$. These modular forms $e_i$ give an analytic parameterization of the two-torsion on an elliptic curve, and one has $M(\Gamma(2)) = \CC[e_1,e_2]$. The $q$-expansions for the $e_i$ are known, and one can use them to show that
\begin{align*}
  e_1 &= \frac{2\pi^2}{3}f_1, & e_2 &= \pi^2\left(-\frac{1}{3}f_1 - 8f_2\right), & e_3 &= \pi^2\left(-\frac{1}{3}f_1 + 8f_2\right).
\end{align*}
\end{ex}

\begin{ex}
\label{e:twodimensionals}
Let $\rho$ be a two-dimensional good irreducible representation of $\SL_2(\ZZ)$. Since $\rho(S)^2 = \pm 1$, necessarily $\Tr(\rho(S))\in \{\pm 2, \pm 2i, 0\}$. If $\Tr(\rho(S)) \neq 0$ then $\rho(S)$ is diagonal and this contradicts the irreducibility of $\rho$. Hence $\Tr(\rho(S)) = 0$. One deduces similarly that $\rho(R)$ must have two distinct sixth roots of unity as eigenvalues. If $\rho$ is even the eigenvalues must be distinct cube roots of unity, while if $\rho$ is odd then they must be two distinct sixth roots of unity that are not cube roots. Thus, if $\cVbar_k(\rho) \cong \cO(k-k_1) \oplus \cO(k-k_2)$, then the multiplicity formulae imply that there are the following possibilities:
\[
\begin{array}{r|c|c|c|c}
\Tr(L) & \Tr(\rho(R))& \Tr(\rho(R^2)) & k_1 & k_2\\
\hline
1/3 & -\zeta-1 & -\zeta & 1 & 3\\
1/2 & -1 & -1 & 2 & 4\\
2/3 & \zeta & \zeta+1 & 3 & 5\\
5/6 & \zeta+1 & -\zeta & 4 & 6\\
1 & 1 & -1 & 5 & 7\\
7/6 & -\zeta & \zeta+1 & 6 & 8\\
4/3 & -\zeta-1 & -\zeta & 7 & 9\\
3/2 & -1 & -1 & 8 & 10 \\
5/3 & \zeta & \zeta + 1 & 9 & 11
\end{array}
\]
The papers \cite{Mason2}, \cite{TubaWenzl} show that all of these possibilities occur. Note that in all cases $\cVbar_k(\rho) \cong \cO(k-6\Tr(L)+1)\oplus \cO(k-6\Tr(L)-1)$, and the weight bound of Proposition \ref{c:weightbound} is sharp. This corresponds to the fact that $M(\rho)$ is a cyclic $M\langle D\rangle$ module in all of these examples, where $M$ is the ring of scalar holomorphic forms of level one, and $D$ is the modular derivative. Note that not all irreducible representations in dimension $2$ are good. For example, the standard representation is not good. Nevertheless, it's not hard to show that one still has $\cVbar_k(\rho) \cong \cO(k-6\Tr(L)+1)\oplus \cO(k-6\Tr(L)-1)$. See \cite{FrancMason} for an explicit description of the corresponding vector valued modular forms.

 The results in \cite{TubaWenzl} can be used to perform a similar analysis in dimensions $3$, $4$ and $5$, although in dimensions $4$ and $5$ there exist noncyclic irreducible examples. See also \cite{Marks} for a discussion of vector valued modular forms in dimensions less than six, and \cite{FrancMason} for a rather detailed description of the case of irreducibles in dimension three.
\end{ex}

\begin{ex}
Let $(V,\rho)$ denote the trace zero subspace of the seven-dimensional permutation representation of $S_7$. It is self-dual, although there is a second six-dimensional irreducible obtained by twisting with the sign character. If we map $T$ to $(17256)(34)$ and $S$ to $(14)(27)(35)$, then we obtain a surjection $\SL_2(\ZZ) \to S_7$, and thus a representation $\rho$ of $\SL_2(\ZZ)$ of dimension $6$ (the other $6$-dimensional irrep is then $\rho \otimes \chi^6$). This representation $\rho$ is known to have noncongruence kernel. In this case one sees that $L$ is conjugate with $\diag(0,\frac 12,\frac 15,\frac 25, \frac 35, \frac 45)$, and both $R$ and $S$ have trace zero in $\rho$. The multiplicity formulae immediately show that
\[
  \cVbar_k(\rho) \cong \cObar(k-2)\oplus 2\cObar(k-4)\oplus 2\cObar(k-6)\oplus\cObar(k-8).
\]
In this example the weight bound of Proposition \ref{c:weightbound} is not sharp.
\end{ex}

\bibliographystyle{plain}

\end{document}